\documentclass[14pt,twoside,a4paper]{article}
\usepackage{mathrsfs}
\usepackage{mathrsfs}
\usepackage{mathrsfs}
\usepackage{amsfonts}
\usepackage{mathrsfs}
\usepackage{amscd}
\usepackage{amsmath}
\usepackage{latexsym}
\usepackage{amsfonts}
\usepackage{amssymb}
\usepackage{amsthm}
\usepackage{graphicx}
\usepackage{color}
\usepackage{cite}
\usepackage{mathrsfs}
\usepackage{mathrsfs}
\usepackage{mathrsfs}
\usepackage{mathrsfs}
\usepackage{mathrsfs}
\usepackage{mathrsfs}
\usepackage{mathrsfs}
\usepackage{stmaryrd}
\usepackage{amsfonts}
\usepackage{amsmath,amssymb}
\usepackage{mathrsfs}
\usepackage{hyperref}

\newtheorem{lemma}{Lemma}[section]
\newtheorem{theorem}{Theorem}[section]
\newtheorem{corollary}{Corollary}[section]

\numberwithin{equation}{section}\allowdisplaybreaks

\def\leq{\leqslant}
\def\geq{\geqslant}

\begin{document}

\title{Standing wave solutions and instability for the Logarithmic Klein--Gordon equation}
\author
         {
           {Lijia Han\footnote{ E--mail address:
	 hljmath@ncepu.edu.cn (L. Han)},\ \ \ {Yue Qiu},\ \ \ {Xiaohong Wang}}\\
   {\small  North China Electric Power University, Beijing 102206, China}\\
         \date{}
         }
\maketitle

\begin{minipage}{13.5cm}
\footnotesize \bf Abstract. \rm In this paper, we study the standing wave solutions of Klein--Gordon equation with logarithmic nonlinearity.
 The existence of the standing wave solution related to the ground state $\phi_0(x)$ is obtained. Further, we prove the instability of  solutions around $\phi_0(x)$.

\vspace{10pt}



\end{minipage}

%
%
%
%
%

\section{Introduction}
\setcounter{section}{1}\setcounter{equation}{0} In this paper, we
study the Cauchy problem for the logarithmic Klein--Gordon equation given by
\begin{align}\label{1.1}
u_{t t}-\Delta u+u=|u|^{p-1} u \ln |u|^2,
\end{align}
with initial data
\begin{align}\label{1.2}
u(0, x)=u_0(x) \in H_r^1, \quad u_t(0, x)=u_1(x) \in L_r^2,
\end{align}
where, $H_r^1=\left\{\phi \mid\right.$ radially symmetric functions on ${R}^{3}\}$, $L_r^2=\left\{\phi \mid\right.$ radially symmetric functions on ${R}^{3}\}$, $u$ is a pair of real--valued functions of $(t, x)$ $\in {R} \times{R}^{3}$, $\Delta$ is the Laplacian on ${R}^{3}$, $2<p<4$.

Equation (1) has been applied in many branches of physics, such as nuclear physics [1], optics [2] and geophysics [3]. It is also introduced  in quantum field [4]. Logarithmic nonlinearity appears naturally in inflation cosmology and in supersymmetric field theories [5-7].
Equation (1) has conservation quantum $E(t)$ as
\begin{align}
E(t) & =\frac{1}{2} \int_{R^3}\left|u_t\right|^2 d x+\frac{1}{2} \int_{R^3}|\nabla u|^2 d x+\frac{1}{2} \int_{R^3}|u|^2 d x+\frac{2}{(p+1)^2} \int_{R^3}|u|^{p+1} d x \\
&\quad-\frac{1}{p+1} \int_{R^3}|u|^{p+1} \operatorname{In} |u|^2 d x.
\end{align}

There are many results about equation (1). When $p=1$, [1] and [3] proved the existence of classical solutions and weak solutions for the logarithmic Klein--Gordon  equation. By using Galerkin method and compactness criterion, [8] proved the existence of global solutions and the blow up of  solutions for the logarithmic Klein--Gordon equation. In [9], the authors studied two regularized finite difference methods for the logarithmic Klein--Gordon equation. In [10],  the authors studied the orbital stability of the standing wave for the logarithmic Klein--Gordon equation. In [11], the author considered the logarithmic Klein--Gordon equation with damping term $u_t$, and obtained the existence of  global solutions  by using the potential well method.  When $p>1$, [12] studied the existence of global solutions and finite--time blow up solutions for logarithmic Klein--Gordon equation. In [13], the author considered equation (1.5),
\begin{equation}\label{1.3.1}
u_{t t}-\Delta u+u=-|u|^{\frac{4}{n-2}} u \log ^\gamma\left(\log \left(10+|u|^2\right)\right),
\end{equation}
for $n= 3, 4, 5$,   $\gamma>0$ is a constant,  the blow up behavior and the global existence of solution were obtained. The logarithmic nonlinearity was also studied in many other equations, such as wave equations and so on. [14] considered the global existence and the finite time blow up of solutions for the semilinear logarithmic wave equation $u_{t t}-\Delta u=|u|^p \operatorname{In}|u|$,
where $1<p<\infty$ if $n=1,2$ and  $2<p<\frac{4}{n-2}$ if $n=3,4,5$.
[15] studied the energy decay estimates and infinite blow up phenomena for a strongly damped semilinear logarithmic wave equation,   $u_{t t}-\Delta u-\Delta u_t=u \log |u|^2$.    [16] concerned with the existence of positive solutions for  logarithmic elliptic equation,  $-\varepsilon^2 \Delta u+V(x) u=u \log |u|^2$, where  $\varepsilon>0$, $N \geq 3$ and $V$ is a continuous function with a global minimum.

In the past decades,  the classical Klein--Gordon system  ($p>1$)
\begin{align}\label{1.3}
& u_{t t}-\Delta u+u=|u|^{p-1} u, \quad(t, x) \in R \times R^N, \\
& u(x, 0)=u_0(x),  u_t(x, 0)=u_1(x), \quad x \in R^N,
\end{align}
had been widely studied by many researchers. For the results about  the existence of local and global solutions, blow up behavior  and scattering  behavior of (1.6), we could refer to [17-25] and the references therein.

 There are also some results about  the standing wave solutions of Klein--Gordon equation. [26] considered the  instability of  standing wave solutions for nonlinear Klein--Gordon equation (1.6). Especially, [27] studied the instability of standing wave solutions for nonlinear Klein--Gordon equation (1.6)  in the critical case in one dimension.
 [28] studied  the stability of the
standing waves of least energy for nonlinear Klein--Gordon equations (1.8),
\begin{align}\label{9.8}
 u_{t t}-\Delta u+u+f(|u|) \arg u=0, \quad x \in R^N, N>2,
\end{align}
where $f$ satisfy $\underset{\eta \rightarrow \infty}{lim} f(\eta) / \eta^l \geq 0$. By using variational  method,  the sufficient conditions of the minimum energy standing wave solutions stability were obtained
 for (1.8). Based on the method in [28], [30] considered the standing wave solutions of nonlinear Klein--Gordon equation (1.9),
\begin{align}\label{1.5.1}
u_{t t}-\Delta u+u=K(x)|u|^{\frac{N}{4}} u,
\end{align}
and discussed the instability of standing wave solutions by establishing some sharp conditions for  blow up and global existence of the solutions, where $K(x)>0$ is a given $C^1$ function such that there is $\frac{N^2}{4}<q<\infty$ such that $K(x) \in L^q\left(R^N\right)$. [31] studied the stability of standing wave solutions  for the coupled nonlinear Klein--Gordon  equation (1.10)
 \begin{equation}\label{9.1}
 \begin{cases}\phi_{t t}-\Delta \phi+\phi=(p+1)|\phi|^{p-1}|\psi|^{q+1} \phi, & t>0, x \in R^N, \\ \psi_{t t}-\Delta \psi+\psi=(q+1)|\psi|^{p-1}|\phi|^{q+1} \psi, & t>0, x \in R^N,\end{cases}
 \end{equation}
 where $(\phi, \psi)$ is a pair of real--valued functions, $N>2$, $0<p, q<\frac{2}{N-2}$ and $p+q<\frac{4}{N}$.
In [32],  the authors studied the instability of standing wave solutions for nonlinear Klein--Gordon equation with a damping term, $u_{t t}-\Delta u+m^2 u+\alpha u_t=|u|^{p-2} u$,
where $p>2$, $ m \neq 0$, $ \alpha>0$.
In [10],  the authors studied the orbital stability of standing wave solutions $u(x, t)=e^{i c t} e^{\frac{1}{2}+\frac{1-c^2}{p}} e^{-\frac{p x^2}{4}}$  for the logarithmic Klein--Gordon equation  (1.11),
\begin{align}\label{1.3.2}
u_{t t}-\Delta u+u=\mu \log \left(|u|^p\right) u,
\end{align}
where $c$ is the frequency, $p>0$, $\mu>0$, $u$ : $\mathbb{R} \times \mathbb{R} \rightarrow \mathbb{C}$  is a complex valued function.

As far as we know, there are no results about the existence and instability for the  standing wave solution related to ground state of logarithmic Klein--Gordon equation  (1.1). Study for the ground state of equation   (1.1) is very important. The ground state solution is the lowest energy state of a system in quantum mechanics. A solution is said to be a ground state if it has the least energy among all solutions.  From physical point of view, it is the most stable configuration of a quantum--mechanical system. Inspired by  [28-29], in this paper we study the existence for the  standing wave solution of equation  (1.1) related to ground state by using the variational method, Gagliardo--Nirenberg inequality, compactness theorem, Sobolev embedding theorem, stretching transformation and so on. Further,  we prove the instability of  solutions around the ground state.

 This paper is organized as follows. In Section 2, we state some preliminaries
 and the main results. In Section 3, we prove the existence of standing wave solutions of equation  (1.1) related to ground state  $\phi_0(x)$ by using the variational method. In Section 4, we prove the instability of  solutions around $\phi_0(x)$.

%
%
%
%
%

\section{Preliminaries and statement of main results}
\setcounter{section}{2}\setcounter{equation}{0}
\subsection{Notations}
We start by giving some definitions of  $H_1^r$ and $L_p^r$,
\begin{align*}
H_r^1=\left\{u \in H^1\left(R^3\right) ; u(x)=u(|x|)\right\}, \quad L_r^p=\left\{u \in L^p\left(R^3\right) ; u(x)=u(|x|), p>2\right\},
\end{align*}
which are radially symmetric spaces with the norm
\begin{align*}
&\|u\|_{H_r^1}^2=\int_{R^3}|\nabla u(x)|^2 d x +\int_{R^3}|u(x)|^2 d x<\infty, \\ &\|u\|_{L_r^p}^p=\int_{R^3}|u(x)|^p d x<\infty.
\end{align*}

Next, we introduce the Gagliardo--Nirenberg inequality in $R^N$, $N>2$,
\begin{align}
\|u\|_{L^r} \leq C\|\nabla u\|_{L^p}^\theta\|u\|_{L^q}^{1-\theta},
\end{align}
where $1<p<N$, $\theta \in[0,1]$, $\frac{1}{r}=\frac{\theta}{p^*}+\frac{1-\theta}{q}$, $p^*=\frac{N p}{N-p}$, $C$ is the best constant.
When $p=q=2$, $r=\alpha$, we get
\begin{align}\label{20.10}
\left\|\phi_k\right\|_{L^\alpha} \leq C_1\left\|\phi_k\right\|_{L^2}^{1-\frac{N(\alpha-2)}{2 \alpha}}\left\|\nabla \phi_k\right\|_{L^2}^{\frac{N(\alpha-2)}{2 \alpha}},
\end{align}
for 2  <$\alpha<\frac{2 N}{N-2}$.

In this paper, $C_i>0(i=1,2,3 \ldots, 20)$ is a constant. For $f(u)=o(u)$  means that  $|f(u) / u| \rightarrow 0$ as $|u| \rightarrow 0$, for $f(u)=O(u)$ means that $|f(u) / u|$  is bounded as $|u| \rightarrow 0$.

\subsection{The main results}
Consider the nonlinear logarithmic Klein--Gordon  equation  (1.1), which has nontrivial standing wave solutions
\begin{align*}
u(x, t)=e^{i \omega t} \phi(x),
\end{align*}
where $\phi(x)$  is a ground state solution of the equation
\begin{align}\label{2.1}
-\Delta \phi+\left(1-\omega^2\right) \phi=|\phi|^{p-1} \phi \operatorname{In} |\phi|^2.
\end{align}

If there is a solution $\phi(x)$ such that the functional $J_\omega(\phi)$ is minimum, then the solution $\phi(x)$ is said to be a ground--state solution of equation  (2.3). Here we define the functional $J_\omega(\phi)$
\begin{align*}
J_\omega(\phi)=\frac{1}{2} \int_{R^3}|\nabla \phi|^2 d x+\frac{1-\omega^2}{2} \int_{R^3}|\phi|^2 d x+\frac{2}{(p+1)^2} \int_{R^3}|\phi|^{p+1} d x-\frac{1}{p+1} \int_{R^3}|\phi|^{p+1} \operatorname{In} |\phi|^2 dx,
\end{align*}
and the functional
$K_\omega(\phi)$
\begin{align*}
K_\omega(\phi) &=\frac{1}{2} \int_{R^3}|\nabla \phi|^2 d x+\frac{3\left(1-\omega^2\right)}{2} \int_{R^3}|\phi|^2 d x+\frac{6}{(p+1)^2} \int_{R^3}|\phi|^{p+1} d x \\
&\quad-\frac{3}{p+1} \int_{R^3}|\phi|^{p+1} \operatorname{In} |\phi|^2 d x.
\end{align*}
In addition, we define the set
\begin{align*}
M_\omega=\left\{\phi \in H_r^1 \mid K_\omega(\phi)=0 ,\phi \neq 0\right\},
\end{align*}
then we consider the constrained variational problem
\begin{align*}
d(\omega)=\inf _{\phi \in M_\omega} J_\omega(\phi).
\end{align*}

First, we have the existence of the standing wave solution related to the ground state $\phi_0(x)$.
 \begin{theorem}\label{thm2.1}
Let $2<p<4$, $N=3$ and $\omega \in[0,1]$,

(i) variational problem
\begin{align}\label{10.1}
d(\omega)=\inf _{\phi \in M_\omega} J_\omega(\phi)
\end{align}
are equivalent to minimization problems
\begin{align}\label{10.2}
\widetilde{d(\omega)}=\inf_{\substack{\phi}} \left\{\frac{1}{3} \int_{R^3}|\nabla \phi|^2 d x, K_\omega(\phi) \leq 0, \phi \neq 0\right\},
\end{align}

(ii) there exists $\phi_0 \in M_\omega$ such that $d(\omega)=\underset{\phi \in M_\omega}{\inf} J_\omega(\phi)=J_\omega\left(\phi_0\right)$,\\

(iii) $\phi_0$ is a solution of $-\Delta \phi+\left(1-\omega^2\right) \phi=|\phi|^{p-1} \phi \ln |\phi|^2$.
\end{theorem}
Second, let
\begin{align*}
X=\left\{u \in H_r^1 \mid \int_{R^3} G(|u|) d x<\infty\right\},
\end{align*}
where
are equivalent to minimization problems
\begin{align}
G(|u|)=\frac{2}{(p+1)^2}|u|^{p+1}-\frac{1}{p+1}|u|^{p+1} \operatorname{In}|u|^2.
\end{align}
We define an invariant set
\begin{align*}
R_1=\left\{u \in X, u_t \in L_r^p \mid E\left(u, u_t\right)<d(0), K_0(u)<0, u \neq 0\right\},
\end{align*}
and obtain  the instability of solutions around $\phi_0$.
\begin{theorem}\label{thm2.2}
Let $\phi(x)$ be a ground state solution of equation  (2.3). For any $\varepsilon>0$, there exists $\left(u_0, u_1\right) \in R_1$ that satisfies
\begin{align}
\left\|\left(u_0, u_1\right)-\left(\phi, 0\right)\right\|_{H_r^1 \times L_r^2}<\varepsilon.
\end{align}

And if the solutions $u(t, x)$ of equation  (1.1)  satisfy $\left(u(0, x), u_t(0, x)\right)=\left(u_0, u_1\right)\in R_1$ and  radially symmetric. Then either

(i) the solutions exist only locally in time $[0,T_0),T_0<\infty$, with $\left\{t_n\right\}$ such that $\left\|u\left(t_n\right)\right\|_{H_r^1} \rightarrow \infty$ as $t_n \rightarrow T_0$  for some sequence $\left\{t_n\right\}$; or

(ii) the solutions exist globally,  then there exists a sequence $\left\{t_n\right\}$ such that $\left\|u\left(t_n\right)\right\|_{H_r^1} \rightarrow \infty$ as $t_n \rightarrow \infty$.
\end{theorem}

\section{Existence of Standing Wave Solutions}
\setcounter{section}{3}\setcounter{equation}{0}
In this section, we prove Theorem 2.1, first, we give some lemmas.
\begin{lemma}$[34]$\label{lem2.1}
	Let $N>2$, under radially symmetric condition, $H^{1}\hookrightarrow L^{p}$      is compact for $2<p<2+4/(N-2)$.
	
\end{lemma}

\begin{lemma}\label{lem2.2}
    If  $\phi_\omega \in H_r^1$ is a solution of $-\Delta \phi+\left(1-\omega^2\right) \phi=|\phi|^{p-1} \phi \operatorname{In} |\phi|^2$,
	then
\begin{align*}
	K_\omega\left(\phi_\omega\right)=0.
\end{align*}
\end{lemma}
\begin{proof} Let $\psi_\beta(x)=\phi_\omega\big(\frac{x}{\beta}\big)$,	we have
\begin{align*}
J_\omega\left(\psi_\beta\right) &=\frac{1}{2} \int_{R^3}\left|\nabla \psi_\beta\right|^2 d x+\frac{1-\omega^2}{2} \int_{R^3}\left|\psi_\beta\right|^2 d x\\
&\quad+\frac{2}{(p+1)^2} \int_{R^3}\left|\psi_\beta\right|^{p+1} d x-\frac{1}{p+1} \int_{R^3}\left|\psi_\beta\right|^{p+1} \operatorname{In} \left|\psi_\beta\right|^2 d x \\
&=\frac{\beta}{2} \int_{R^3}\left|\nabla \phi_\omega\right|^2 d x+\frac{\beta^3\left(1-\omega^2\right)}{2} \int_{R^3}\left|\phi_\omega\right|^2 d x \\
&\quad+\frac{2 \beta^3}{(p+1)^2} \int_{R^3}\left|\phi_\omega\right|^{p+1} d x-\frac{\beta^3}{p+1} \int_{R^3}\left|\phi_\omega\right|^{p+1} \operatorname{In} \left|\phi_\omega\right|^2 d x.
\end{align*}
 Since  $\phi_\omega$ is a solution of equation (2.3), then
\begin{align*}
\delta J_\omega\left(\phi_\omega\right)=\left.0 \Rightarrow \frac{d\left(J_\omega\left(\psi_\beta\right)\right)}{d \beta}\right|_{\beta=1}=0,
\end{align*}
so
\begin{align*}
\left.\frac{d\left(J_\omega\left(\phi_\beta\right)\right)}{d \beta}\right|_{\beta=1} & =\frac{1}{2} \int_{R^3}|\nabla \phi|^2 d x+\frac{3\left(1-\omega^2\right)}{2} \int_{R^3}|\phi|^2 d x \\
& \quad+\frac{6}{(p+1)^2} \int_{R^3}|\phi|^{p+1} d x-\frac{3}{p+1} \int_{R^3}|\phi|^{p+1} \operatorname{In} |\phi|^2 d x \\
& =K_\omega\left(\phi_\omega\right)=0.
\end{align*}
This completes the proof of the Lemma 3.2.
\end{proof}
 \noindent {\bf Proof of Theorem 2.1.} We divide the proof of  Theorem 2.1 into three steps.\\		 		
Step1: Let $\widetilde{d(\omega)}=\inf \left\{\frac{1}{3} \int_{R^3}|\nabla \phi|^2 d x, K_\omega(\phi) \leq 0, \phi \neq 0\right\}$, we prove the equivalence of $d(\omega)$ and  $\widetilde{d(\omega)}$.
By the definition of $d(\omega)$ and  $\widetilde{d(\omega)}$, we have  $\widetilde{d(\omega)} \leq d(\omega)$. For any $\phi \in H_r^1$ satisfying  $ K_\omega(\phi)<0$,
let $\psi_\beta(x)=\phi\big(\frac{x}{\beta}\big)$,
\begin{align*}
K_\omega\left(\psi_\beta\right) &=\frac{1}{2} \int_{R^3}\left|\nabla \psi_\beta\right|^2 d x+\frac{3\left(1-\omega^2\right)}{2} \int_{R^3}\left|\psi_\beta\right|^2 d x \\
&\quad+\frac{6}{(p+1)^2} \int_{R^3}\left|\psi_\beta\right|^{p+1} d x-\frac{3}{p+1} \int_{R^3}\left|\psi_\beta\right|^{p+1} \operatorname{In} \left|\psi_\beta\right|^2 d x \\
&=\frac{\beta}{2} \int_{R^3}|\nabla \phi|^2 d x+\frac{3\beta^3 \left(1-\omega^2\right)}{2} \int_{R^3}|\phi|^2 d x \\
&\quad+\frac{6 \beta^3}{(p+1)^2} \int_{R^3}|\phi|^{p+1} d x-\frac{3 \beta^3}{p+1} \int_{R^3}|\phi|^{p+1} \operatorname{In} |\phi|^2 d x.
\end{align*}
When $\beta=1$, $K_\omega\left(\psi_1\right)=K_\omega(\phi)<0$.\\
When $\beta \rightarrow 0$,
\begin{align*}
K_\omega\left(\psi_\beta\right) & =\beta\left[\frac{1}{2} \int_{R^3}|\nabla \phi|^2 d x+\frac{3\beta^2\left(1-\omega^2\right)}{2} \int_{R^3}|\phi|^2 d x\right. \\
& \quad\left.+\frac{6 \beta^2}{(p+1)^2} \int_{R^3}|\phi|^{p+1} d x-\frac{3 \beta^2}{p+1} \int_{R^3}|\phi|^{p+1} \operatorname{In} |\phi|^2 d x\right]>0.
\end{align*}
So there exists  $\beta^* \in(0,1)$ such that $K_\omega\left(\psi_{\beta^*}\right)=0$ such as
\begin{align*}
\frac{1}{3} \int_{R^3}\left|\nabla \psi_{\beta^*}\right|^2 d x=\frac{\beta^*}{3} \int_{R^3}|\nabla \phi|^2 d x<\frac{1}{3} \int_{R^3}|\nabla \phi|^2 d x,
\end{align*}
which implies
$d(\omega) \leq \widetilde{d(\omega)}$. Thus, $d(\omega) = \widetilde{d(\omega)}$ holds.\\
Step2: We prove there exists $\phi_0 \in M_\omega$  such that
\begin{align*}
d(\omega)=\inf _{\phi \in M_\omega} J_\omega(\phi)=J_\omega\left(\phi_0\right).
\end{align*}
Let $\phi_k$ be a minimum sequence for the minimization problem (2.5),
then $\int_{R^3}\left|\nabla \phi_k\right|^2 d x$ is bounded. Since $K_\omega\left(\phi_k\right) \leq 0$, then
\begin{align}\label{20.16}
K_\omega\left(\phi_k\right) &=\frac{1}{2} \int_{R^3}\left|\nabla \phi_k\right|^2 d x+\frac{3\left(1-\omega^2\right)}{2} \int_{R^3}\left|\phi_k\right|^2 d x \\
&\quad+\frac{6}{(p+1)^2} \int_{R^3}\left|\phi_k\right|^{p+1} d x-\frac{3}{p+1} \int_{R^3}\left|\phi_k\right|^{p+1} \operatorname{In} \left|\phi_k\right|^2 d x \leq 0,
\end{align}
this indicates
\begin{align}\label{2.4}
&\frac{1}{2} \int_{R^3}\left|\nabla \phi_k\right|^2 d x+\frac{3\left(1-\omega^2\right)}{2} \int_{R^3}\left|\phi_k\right|^2 d x \\
&\leq \frac{3}{p+1} \int_{R^3}\left|\phi_k\right|^{p+1} \operatorname{In} \left|\phi_k\right|^2 d x+\frac{6}{(p+1)^2} \int_{R^3}\left|\phi_k\right|^{p+1} d x.
\end{align}
Notice that $u \operatorname{In} u^2 \leq C_2\left(1+u^2\right)$, $u>0$, then
\begin{align}\label{8.1}
&\frac{1}{2} \int_{R^3}\left|\nabla \phi_k\right|^2 d x+\frac{3\left(1-\omega^2\right)}{2} \int_{R^3}\left|\phi_k\right|^2 d x \\
&\leq \frac{3C_2}{p+1} \int_{R^3}\left|\phi_k\right|^p d x+\frac{3C_2}{p+1} \int_{R^3}\left|\phi_k\right|^{p+2} d x+\frac{6}{(p+1)^2} \int_{R^3}\left|\phi_k\right|^{p+1} d x.
\end{align}
By using the Gagliardo--Nirenberg inequality, we obtain
\begin{align}\label{30.1}
\left\|\phi_k\right\|_{L_r^p} \leq C_3\left\|\phi_k\right\|_{L_r^2}^{1-\frac{3(p-2)}{2 p}}\left\|\nabla \phi_k\right\|_{L_r^2}^{\frac{3(p-2)}{2 p}},
\end{align}
\begin{align}\label{30.2}
\left\|\phi_k\right\|_{L_r^{p+1}} \leq C_4\left\|\phi_k\right\|_{L_r^2}^{1-\frac{3(p-1)}{2(p+1)}}\left\|\nabla \phi_k\right\|_{L_r^2}^{\frac{3(p-1)}{2(p+1)}},
\end{align}
\begin{align}\label{30.3}
\left\|\phi_k\right\|_{L_r^{p+2}} \leq C_5\left\|\phi_k\right\|_{L_r^2}^{1-\frac{3p}{2(p+2)}}\left\|\nabla \phi_k\right\|_{L_r^2}^{\frac{3p}{2(p+2)}},
\end{align}
where $2<p<4$. Then from Young's inequality in (3.7)--(3.9), we have
\begin{align}
\left\|\phi_k\right\|_{L_r^p}^p \leq C_6\left\|\phi_k\right\|_{L_r^2}^{p-\frac{3(p-2)}{2}}\left\|\nabla \phi_k\right\|_{L_r^2}^{\frac{3(p-2)}{2}} \leq C_6\left(\varepsilon\left\|\phi_k\right\|_{L_r^2}^{\frac{m(6-p)}{2}}+\varepsilon^{-\frac{n}{m}}\left\|\nabla \phi_k\right\|_{L_r^2}^{\frac{3 n(p-2)}{2}}\right),
\end{align}
\begin{align}
\left\|\phi_k\right\|_{L_r^{p+1}}^{p+1} \leq C_7\left\|\phi_k\right\|_{L_r^2}^{p+1-\frac{3(p-1)}{2}}\left\|\nabla \phi_k\right\|_{L_r^2}^{\frac{3(p-1)}{2}} \leq C_7\left(\varepsilon\left\|\phi_k\right\|_{L_r^2}^{\frac{m(5-p)}{2}}+\varepsilon^{-\frac{n}{m}}\left\|\nabla \phi_k\right\|_{L_r^2}^{\frac{3 n(p-1)}{2}}\right),
\end{align}
\begin{align}\label{10.3}
\left\|\phi_k\right\|_{L_r^{+2}}^{p+2} \leq C_8\left\|\phi_k\right\|_{L_r^2}^{p+2-\frac{3 p}{2}}\left\|\nabla \phi_k\right\|_{L_r^2}^{\frac{3 p}{2}} \leq C_8\left(\varepsilon\left\|\phi_k\right\|_{L_r^2}^{\frac{m(4-p)}{2}}+\varepsilon^{-\frac{n}{m}}\left\|\nabla \phi_k\right\|_{L_r^2}^{\frac{3 n p}{2}}\right),
\end{align}
where $\varepsilon$ is sufficiently small, $m$, $n$ satisfy $m>1$, $n>1$, $\frac{1}{m}+\frac{1}{n}=1$. If $\left\|\phi_k\right\|_{L_r^2}<1$, $\left\|\phi_k\right\|_{H_r^1}$ is obviously bounded. If $\left\|\phi_k\right\|_{L_r^2}>1$, from (3.7)--(3.12), we have
\begin{align}\label{10.9}
& \frac{1}{2}\left\|\nabla \phi_k\right\|_{L_r^2}^2+\frac{3\left(1-\omega^2\right)}{2}\left\|\phi_k\right\|_{L_r^2}^2-\frac{C_9 \varepsilon\left\|\phi_k\right\|_{L_r^2}^{\frac{m(6-p)}{2}}}{p+1}-\frac{C_9 \varepsilon\left\|\phi_k\right\|_{L_r^2}^{\frac{m(4-p)}{2}}}{p+1}-\frac{C_9 \varepsilon\left\|\phi_k\right\|_{L_r^2}^{\frac{m(5-p)}{2}}}{(p+1)^2} \\
& \leq \frac{C_9 \varepsilon^{-\frac{n}{m}}\left\|\nabla \phi_k\right\|_{L_r^2}^{\frac{3 n(p-2)}{2}}}{p+1}+\frac{C_9 \varepsilon^{-\frac{n}{m}}\left\|\nabla \phi_k\right\|_{L_r^2}^{\frac{3 n p}{2}}}{p+1}+\frac{C_9 \varepsilon^{-\frac{n}{m}}\left\|\nabla \phi_k\right\|_{L_r^2}^{\frac{3 n(p-1)}{2}}}{(p+1)^2} .
\end{align}
Choose $1<m<\frac{4}{6-p}$, we obtain
\begin{align*}
\max \left\{2, \frac{m(6-p)}{2}, \frac{m(5-p)}{2}, \frac{m(4-p)}{2}\right\}=2,
\end{align*}
so (3.13) can be written as
\begin{align}\label{8.3}
\frac{3\left(1-\omega^2\right)}{2}\left\|\phi_k\right\|_{L_r^2}^2 \leq C_{9}\left\|\nabla \phi_k\right\|_{L_r^2}^{\frac{3 n(p-2)}{2}}+C_{9}\left\|\nabla \phi_k\right\|_{L_r^2}^{\frac{3 n p}{2}}+C_{9}\left\|\nabla \phi_k\right\|_{L_r^2}^{\frac{3 n(p-1)}{2}}.
\end{align}
Since $\int_{R^3}\left|\nabla \phi_k\right|^2 d x$ is bounded, then $\left\|\phi_k\right\|_{H_r^1}$ is bounded by (3.15), so there must be a convergent subsequence, also denoted as $\phi_k$, such that
\begin{align}\label{2.8}
\phi_k \rightarrow \phi_0 \text { weakly in } H_r^1,
\end{align}
from Lemma 3.1,
\begin{align}\label{2.8.1}
\phi_k \rightarrow \phi_0 \text { strongly in } L_r^p,  \quad 2<p<6.
\end{align}
By the lower semicontinuity of the limit, we have
\begin{align}
K_\omega\left(\phi_0\right)= & \frac{1}{2} \int_{R^3}\left|\nabla \phi_0\right|^2 d x+\frac{3\left(1-\omega^2\right)}{2} \int_{R^3}\left|\phi_0\right|^2 d x \\
& +\frac{6}{(p+1)^2} \int_{R^3}\left|\phi_0\right|^{p+1} d x-\frac{3}{p+1} \int_{R^3}\left|\phi_0\right|^{p+1} \operatorname{In}\left|\phi_0\right|^2 d x \\
\leq & \varliminf _{k \rightarrow \infty}\left(\frac{1}{2} \int_{R^3}\left|\nabla \phi_k\right|^2 d x+\frac{3\left(1-\omega^2\right)}{2} \int_{R^3}\left|\phi_k\right|^2 d x\right. \\
& \left.+\frac{6}{(p+1)^2} \int_{R^3}\left|\phi_k\right|^{p+1} d x-\frac{3}{p+1} \int_{R^3}\left|\phi_k\right|^{p+1} \operatorname{In}\left|\phi_k\right|^2 d x\right) \\
= & \varliminf _{k \rightarrow \infty} K_\omega\left(\phi_k\right) \leq 0,
\end{align}
according to Step1, we can obtain $K_\omega\left(\phi_0\right)=0$. Next we prove $\phi_0 \neq 0$. By contradiction, if $\phi_0=0$, we get
\begin{align}\label{10.4}
\phi_k \rightarrow 0 \in H_r^1, \quad \phi_k \rightarrow 0 \in L_r^p, \quad  2<p<6.
\end{align}
However, by Sobolev embedding theorem, when $2 \leq p \leq 4$, we get
\begin{align}
\left\|\phi_k\right\|_{L_r^p}^p \leq C_{10}\left\|\phi_k\right\|_{H_r^1}^p,\quad \left\|\phi_k\right\|_{L_r^p}^{p+1} \leq C_{11}\left\|\phi_k\right\|_{H_r^1}^{p+1},\quad \left\|\phi_k\right\|_{L_r^p}^{p+2} \leq C_{12}\left\|\phi_k\right\|_{H_r^1}^{p+2}.
\end{align}
So, from (3.5) we obtain
\begin{align}
\alpha\left\|\phi_k\right\|_{H_r^1}^2 \leq \frac{C_{13} }{p+1}\left\|\phi_k\right\|_{H_r^1}^p+\frac{C_{13} }{p+1}\left\|\phi_k\right\|_{H_r^1}^{p+2}+\frac{C_{13} }{(p+1)^2}\left\|\phi_k\right\|_{H_r^1}^{p+1},
\end{align}
where $\alpha=\min \left\{\frac{1}{2}, \frac{3\left(1-\omega^2\right)}{2}\right\}$.
Since $\phi_k \neq 0$ and $\left\|\phi_k\right\|_{H_r^1}$ is bounded, then
 \begin{align}\label{20.15}
\left\|\phi_k\right\|_{H_r^1}^{p-2}+\left\|\phi_k\right\|_{H_r^1}^p+\left\|\phi_k\right\|_{H_r^1}^{p-1} \geq \frac{\alpha(p+1)}{C_{13}},
\end{align}
which contradicts to (3.23). Therefore $\phi_0 \neq 0$ and combine with $K_\omega\left(\phi_0\right)=0$, we deduce that $\phi_0 \in M_\omega$. By the lower semicontinuity of the limit, we have
\begin{align}
d(\omega)=\inf _{\phi \in M_\omega} J_\omega(\phi) \leq J_\omega\left(\phi_0\right)=\frac{1}{3} \int_{R^3}\left|\nabla \phi_0\right|^2 d x \leq \varliminf_{k \rightarrow \infty} \frac{1}{3} \int_{R^3}\left|\nabla \phi_k\right|^2 d x=d(\omega).
\end{align}
Step3: We prove  $\phi_0$ is a solution of $-\Delta \phi+\left(1-\omega^2\right) \phi=|\phi|^{p-1} \phi \operatorname{In} |\phi|^2$.

By step2, we consider $\inf J_\omega(\phi)$       under condition $K_\omega\left(\phi\right)=0$.
Using lagrange multiplier method, we have
\begin{align}\label{10.10}
\delta J_\omega\left(\phi_0\right)=\lambda \delta K_\omega\left(\phi_0\right),
\end{align}
that is
\begin{align}
-(1-\lambda) \Delta \phi_0+(1-3\lambda )\left(1-\omega^2\right) \phi_0-(1-3\lambda )\left|\phi_0\right|^p \operatorname{In}\left|\phi_0\right|^2=0.
\end{align}
Let
\begin{align}
J_{1_\omega}(\phi)= & \frac{1-\lambda}{2} \int_{R^3}|\nabla \phi|^2 d x+\frac{(1-3\lambda)\left(1-\omega^2\right)}{2} \int_{R^3}|\phi|^2 d x \\
& +\frac{2(1-3\lambda)}{(p+1)^2} \int_{R^3}|\phi|^{p+1} d x-\frac{1-3\lambda }{p+1} \int_{R^3}|\phi|^{p+1} \operatorname{In} |\phi|^2 d x,
\end{align}
similar to the proof of Lemma 3.2, we obtain
\begin{align}\label{2.13}
&\frac{1}{2} \int_{R^3}\left|\nabla \phi_0\right|^2 d x+\frac{3\left(1-\omega^2\right)}{2} \int_{R^3}\left|\phi_0\right|^2 d x+\frac{6}{(p+1)^2} \int_{R^3}\left|\phi_0\right|^{p+1} d x \\
&\quad-\frac{3}{p+1} \int_{R^3}\left|\phi_0\right|^{p+1} \operatorname{In} \left|\phi_0\right|^2 d x \\
&=\frac{\lambda}{2} \int_{R^3}\left|\nabla \phi_0\right|^2 d x+\frac{9\lambda \left(1-\omega^2\right)}{2} \int_{R^3}\left|\phi_0\right|^2 d x \\
&\quad+\frac{18\lambda}{(p+1)^2} \int_{R^3}\left|\phi_0\right|^{p+1} d x-\frac{9\lambda}{p+1} \int_{R^3}\left|\phi_0\right|^{p+1} \operatorname{In} \left|\phi_0\right|^2 d x.
\end{align}
Thanks to $K_\omega\left(\phi_0\right)=0$ , we show that 3$\lambda K_\omega\left(\phi_0\right)=0$,
\begin{align}\label{2.14}
&\frac{3\lambda}{2} \int_{R^3}\left|\nabla \phi_0\right|^2 d x+\frac{9\lambda \left(1-\omega^2\right)}{2} \int_{R^3}\left|\phi_0\right|^2 d x \\
&+\frac{18\lambda}{(p+1)^2} \int_{R^3}\left|\phi_0\right|^{p+1} d x-\frac{9\lambda}{p+1} \int_{R^3}\left|\phi_0\right|^{p+1} \operatorname{In} \left|\phi_0\right|^2 d x=0.
\end{align}
From (3.32)--(3.36), we have
\begin{align*}
-\lambda \int_{R^3}\left|\nabla \phi_0\right|^2 d x=0.
\end{align*}
It implies $\lambda=0$, from (3.28) we obtain $J_\omega\left(\phi_0\right)=0$,  therefore $\phi_0$ is a solution of $-\Delta \phi+\left(1-\omega^2\right) \phi=|\phi|^{p-1} \phi \operatorname{In} |\phi|^2$. This completes the proof of Theorem 2.1.$\hfill\Box$\\

From Theorem 2.1, we get the following corollary.
\begin{corollary}\label{cor2.2}	
Let $G_\omega$ be the solution set of (2.4), $G_\omega$  is also the solution set of problem (3.38),
\begin{equation}\label{10.5}
\inf J_\omega(\phi)=d(\omega), \quad \omega \in(0,1), \quad \frac{1}{3} \int_{R^3}|\nabla \phi|^2 d x=d(\omega).
\end{equation}
\end{corollary}
\begin{proof}
Assuming that $\phi$ satisfies
\begin{align}\label{10.6}
\frac{1}{3} \int_{R^3}|\nabla \phi|^2 d x=d(\omega), \quad J_\omega(\phi)<d(\omega),
\end{align}
then
\begin{align*}
\frac{1}{3} K_\omega(\phi)=J_\omega(\phi)-\frac{1}{3} \int_{R^3}|\nabla \phi|^2 d x=J_\omega(\phi)-d(\omega)<0.
\end{align*}
According to Theorem 2.1,
\begin{align*}
d(\omega)=\inf \left\{\frac{1}{3} \int_{R^3}|\nabla \phi|^2 d x, K_\omega(\phi) \leq 0, \phi \neq 0\right\},
\end{align*}
we obtain $d(\omega)<\frac{1}{3} \int_{R^3}|\nabla \phi|^2 d x$,  which contradicts to (3.39), then
\begin{align*}
\inf J_\omega(\phi)=d(\omega), \quad \frac{1}{3} \int_{R^3}|\nabla \phi|^2 d x=d(\omega).
\end{align*}
Next, we prove that $G_\omega$ is also the solution set of problem (3.38). Let $\phi$  be the solution to problem (3.38), then
\begin{align}\label{10.7}
\delta J_\omega(\phi)+\lambda \delta \int_{R^3}|\nabla \phi|^2 d x=0.
\end{align}
We calculate the (3.40) and obtain
\begin{align*}
\int_{R^3}-(1+\lambda) \Delta \phi+\left(1-\omega^2\right) \phi-|\phi|^{p-1} \phi \operatorname{In}|\phi|^2 d x=0.
\end{align*}
Let
\begin{align}
J_{2 \omega}(\phi)= & \frac{1+\lambda}{2} \int_{R^3}|\nabla \phi|^2 d x+\frac{1-\omega^2}{2} \int_{R^3}|\phi|^2 d x \\
& +\frac{2}{(p+1)^2} \int_{R^3}|\phi|^{p+1} d x-\frac{1}{p+1} \int_{R^3}|\phi|^{p+1} \operatorname{In}|\phi|^2 d x,
\end{align}
similar to the proof of Lemma 3.2, we obtain
\begin{align}\label{10.8}
& \frac{1+\lambda}{2} \int_{R^3}|\nabla \phi|^2 d x+\frac{3\left(1-\omega^2\right)}{2} \int_{R^3}|\phi|^2 d x \\
& +\frac{6}{(p+1)^2} \int_{R^3}|\phi|^{p+1} d x-\frac{3}{p+1} \int_{R^3}|\phi|^{p+1} \mathrm{In}|\phi|^2 d x=0.
\end{align}
According to (3.43) and the definition of $J_{\omega}(\phi)$, we get
\begin{align*}
J_\omega(\phi)=\frac{2-\lambda}{6} \int_{R^3}|\nabla \phi|^2 d x=d(\omega)=\frac{1}{3} \int_{R^3}|\nabla \phi|^2 d x,
\end{align*}
therefore $\lambda=0$. From (3.43), we obtain $K_\omega(\phi)=0$, $\phi \in G_\omega$. This completes the proof of the Corollary 3.1.
\end{proof}

%
%
%
%
%

\section{The Instability of Standing Wave Solutions}
\setcounter{section}{4}\setcounter{equation}{0}
In this section, we consider the Cauchy problem for equation (1.1),
and define the energy of equation (1.1) as
\begin{align*}
E\left(u, u_t\right)=\frac{1}{2} \int_{R^3}\left|u_t\right|^2 d x+J_0(u).
\end{align*}
where \begin{align*}
J_0(\phi)=\frac{1}{2} \int_{R^3}|\nabla \phi|^2 d x+\frac{1}{2} \int_{R^3}|\phi|^2 d x+\frac{2}{(p+1)^2} \int_{R^3}|\phi|^{p+1} d x-\frac{1}{p+1} \int_{R^3}|\phi|^{p+1} \operatorname{In} |\phi|^2 dx.
\end{align*}
Fristly, we give the following lemmas.
\begin{lemma}\label{lem5.1}
If the solutions of Cauchy problem (1.1)--(1.2) satisfy $E\left(u_0, u_1\right)=E\left(u, u_t\right)$,
then $R_1$ is an invariant set under the flow of problems (1.1)--(1.2).
\end{lemma}
\begin{proof}
	Let $\left(u_0, u_1\right) \in R_1$,  according to the conservation of energy, we obtain
	\begin{align}
	E\left(u_0, u_1\right)=E\left(u, u_t\right)<d(0).
	\end{align}
Assume there exists a  $t_1$ such that $\left(u\left(t_1\right), u_t\left(t_1\right)\right) \notin R_1$, then $K\left(u\left(t_1\right)\right) \geq 0$. From the lower semi--continuity of $K(u(t))$, there is a maximum value $t_0$, which satisfies $K\left(u\left(t_0\right)\right) \geq 0$. So for $t \in\left(t_0,+\infty\right)$, we have
$K(u(t))<0$,
\begin{align}\label{5.1.4}
& \frac{1}{3} \int_{R^3}\left|\nabla u\left(t_0\right)\right|^2 d x \leq \varliminf_{\substack{t \rightarrow t_0 \\
		t<t_0}} \frac{1}{3} \int_{R^3}|\nabla u(t)|^2 d x \\
& \leq \varliminf_{\substack{t \rightarrow t_0  \\
		t<t_0}}\left(\frac{1}{3} \int_{R^3}|\nabla u(t)|^2 d x+\frac{1}{3} K(u(t))\right) \\
& =\varliminf_{\substack{t \rightarrow t_0 \\
		t<t_0}} J(u(t)) \leq \varliminf_{\substack{t \rightarrow t_0 \\
		t<t_0}}\left(J(u(t))+\frac{1}{2} \int_{R^3}\left|u_t(t)\right|^2 d x\right) \\
& =\varliminf_{\substack{t \rightarrow t_0 \\
		t<t_0}} E\left(u, u_t\right)<d(0),
\end{align}
which contradicts the definition of $d(0)$, so $R_1$ is an invariant set under the flow of Cauchy problem (1.1)--(1.2).
\end{proof}
\begin{lemma}$[34]$\label{lem5.2}
	let $u \in H_r^1$, then \begin{align*}
	|u(x)| \leq C_N|x|^{\frac{1-N}{2}}\|u\|_{H_r^1},
	\end{align*}
	where, $C_N$ only depends on $N$.
\end{lemma}
\begin{lemma}\label{lem5.3}
Let \begin{align*}
\phi_m(r, t) & =r H(m-r)+\frac{\operatorname{In}(m) r}{\operatorname{In}(r)}[H(r-m)-H(r-(t+2 m))] \\
&\ \  -\frac{\operatorname{In}(m)(t+2 m)}{\operatorname{In}(t+2 m)}[H(r-(t+2 m))],
\end{align*}
where, if $s \geq 0$, $H(s)=1$, if $s<0$, $H(s)=0$. Then we obtain
\begin{align}\label{5.1.6}
& \operatorname{Re} \int_{R^3} u_t \phi_m \bar{u}_r d x-\operatorname{Re} \int_{R^3} u_t(0) \phi_m(r, 0) \bar{u}_r(0) d x \\
& =\int_0^t \int_{r<m}-\frac{3}{2}\left|u_t\right|^2+\frac{1}{2}|\nabla u|^2+\frac{3}{2}|u|^2+3 G(|u|) d x d s \\
& \quad+\frac{\operatorname{In} m}{\operatorname{In} r} \int_0^t \int_{m<r<2 m+t}-\frac{3}{2}\left|u_t\right|^2+\frac{1}{2}|\nabla u|^2+\frac{3}{2}|u|^2+3 G(|u|) d x d s \\
& \quad+\frac{\operatorname{In} m}{(\operatorname{In} r)^2} \int_0^t \int_{m<r<2 m+t} \frac{1}{2}\left|u_t\right|^2+\frac{1}{2}|\nabla u|^2-\frac{1}{2}|u|^2-G(|u|) d x d s \\
& \quad+\int_0^t \int_{r>2 m+t}-\operatorname{Re} e_{1 m} \bar{u}_r u_t+e_{2 m}\left[-\frac{1}{2}\left|u_t\right|^2+|\nabla u|^2-|u|^2-2 G(|u|)\right] d x d s .
\end{align}
\end{lemma}
\begin{proof}
We multiply equation (1.1) by $\phi_m(r, t) \bar{u}_r$, then integrate on $R^3$, take the real part, we obtain
\begin{align}\label{5.1.5}
& \frac{d}{d t} \operatorname{Re} \int_{R^3} u_t \phi_m \bar{u}_r d x-\operatorname{Re} \int_{R^3} u_t \frac{d}{d t}\left(\phi_m\right) \bar{u}_r d x-\operatorname{Re} \int_{R^3} u_t \phi_m \frac{d}{d t}\left(\bar{u}_r\right) d x \\
& -\operatorname{Re} \int_{R^3} \Delta u \phi_m \bar{u}_r d x+\operatorname{Re} \int_{R^3} u \phi_m \bar{u}_r d x+\operatorname{Re} \int_{R^3} f(|u|) \frac{u}{|u|} \phi_m \bar{u}_r d x=0,
\end{align}
where, $f(|u|)=-|u|^p \operatorname{In}|u|^2$.

 Let $G^{\prime}(|u|)=f(|u|)$,
we  calculate equation (4.11) and obtain
\begin{align*}
\frac{d}{d t} \operatorname{Re} & \int_{R^3} u_t \phi_m \bar{u}_r d x+\int_{r<m} \frac{3}{2}\left|u_t\right|^2-\frac{1}{2}|\nabla u|^2-\frac{3}{2}|u|^2-3 G(|u|) d x \\
& +\frac{\operatorname{In} m}{\operatorname{In} r}\left[\int_{m<r<2 m+t} \frac{3}{2}\left|u_t\right|^2-\frac{1}{2}|\nabla u|^2-\frac{3}{2}|u|^2-3 G(|u|) d x\right] \\
& -\frac{\operatorname{In} m}{(\operatorname{In} r)^2}\left[\int_{m<r<2 m+t} \frac{1}{2}\left|u_t\right|^2+\frac{1}{2}|\nabla u|^2-\frac{1}{2}|u|^2-G(|u|) d x\right] \\
& +\int_{r>2 m+t} \operatorname{Re} e_{1 m} \bar{u}_r  u_t+e_{2 m}\left[\frac{1}{2}\left|u_t\right|^2-|\nabla u|^2+|u|^2+2 G(|u|)\right] d x=0 ,
\end{align*}
where $e_{i m}$ $(i=1,2)$  is a constant and $\left|e_{i m}\right| \leq C_{17}$ for $r>2 m+t$. Integrating with respect to $t$ we obtain (4.6).

\end{proof}
\noindent {\bf Proof of Theorem 2.2.} We prove Theorem 2.2  by contradiction. Statement (i) is just a consequence of the local existence of the solution. Now we proof (ii). Assuming $u(t, x)$ is bounded in $H_r^1\left(R^3\right)$, then there exists an $\varepsilon>0$ such that $K_0(u(t))<-\varepsilon$ for any $t$. Otherwise, since $R_1$ is an invariant set, there exists an sequence $\left\{t_k\right\}$ that satisfies $K_0\left(u\left(t_k\right)\right) \rightarrow 0$, so there is a convergent subsequence $u\left(t_k\right)$,
\begin{align*}
u\left(t_k\right) \rightarrow u^* \text { weakly in } H_r^1\left(R^3\right),
\end{align*}
from Lemma 3.1, we have
\begin{align*}
u\left(t_k\right) \rightarrow u^* \text { strongly in } L_r^p\left(R^3\right), \quad 2<p<6,
\end{align*}
then we get
\begin{align}
\frac{1}{3} \int_{R^3}\left|\nabla u^*\right|^2 d x & \leq \varliminf_{k \rightarrow \infty} \frac{1}{3} \int_{R^3}\left|\nabla u\left(t_k\right)\right|^2 d x \\
& =\varliminf_{k \rightarrow \infty}\left[J_0\left(u\left(t_k\right)\right)-\frac{1}{3} K\left(u\left(t_k\right)\right)\right] \\
& =\varliminf_{k \rightarrow \infty} J_0\left(u\left(t_k\right)\right) \\
& \leq \varliminf_{k \rightarrow \infty}\left[J_0\left(u\left(t_k\right)\right)+\frac{1}{2} \int_{R^3}\left|u_t\left(t_k\right)\right|^2 d x\right] \\
& =E\left(u_0, u_1\right)<d,
\end{align}
which contradicts the definition of $d(0)$. Therefore such an $\varepsilon>0$ exists.
  According to Lemma 4.2, we can get that
  \begin{align*}
  |u(r, t)| \leq C_N\|u(t)\|_{H_r^1} r^{-1}<C_{18} r^{-1},
  \end{align*}
when $r$ is sufficiently large,  we have $|u(r, t)| \leq 1$, so \begin{align*}
G(|u|)=\frac{2}{(p+1)^2}|u|^{p+1}-\frac{1}{(p+1)}|u|^{p+1} \operatorname{In}|u|^2>0.
\end{align*}
We estimate the (4.6) and  obtain
\begin{align}\label{5.1.7}
\int_{m<r<2 m+t} \frac{\operatorname{In} m}{2 \operatorname{In} r}|\nabla u|^2 d x&+\int_{m<r<2 m+t} \frac{\operatorname{In} m}{2(\operatorname{In} r)^2}|\nabla u|^2 d x \\
&<\frac{1}{2} \int_{m<r<2 m+t}|\nabla u|^2 d x+\frac{1}{2 \operatorname{In} m} \int_{m<r<2 m+t}|\nabla u|^2 d x,
\end{align}
\begin{align}
\int_{m<r<2 m+t} \frac{3 \operatorname{In} m}{2 \operatorname{In} r}|u|^2 d x-\int_{m<r<2 m+t} \frac{\operatorname{In} m}{2(\operatorname{In} r)^2}|u|^2 d x<\frac{3}{2} \int_{m<r<2 m+t}|u|^2 d x ,
\end{align}
\begin{align}\label{5.1.8}
\int_{m<r<2 m+t} \frac{3 \operatorname{In} m}{\operatorname{In} r} G(|u|) d x-\int_{m<r<2 m+t} \frac{\operatorname{In} m}{(\operatorname{In} r)^2} G(|u|) d x<3 \int_{m<r<2 m+t} G(|u|) d x.
\end{align}
Finally, since the radially symmetric solution is strong outside the optical cone containing the origin, at the base of the optical cone, we have the following estimates,
\begin{align}
 \int_{r>2 m+t}-\operatorname{Re} e_{1 m} \bar{u}_r u_t&+e_{2 m}\left[-\frac{1}{2}\left|u_t\right|^2+|\nabla u|^2-|u|^2-2 G(|u|)\right] d x \\
& \leq C_{19} \int_{r>2 m+t}\left|\bar{u}_r u_t\right|-\frac{1}{2}\left|u_t\right|^2+|\nabla u|^2-|u|^2-2 G(|u|) d x \\
& \leq C_{19} \int_{r>2 m+t}\left|\bar{u}_r u_t\right|+\frac{1}{2}\left|u_t\right|^2+|\nabla u|^2+|u|^2+2 G(|u|) d x \\
& \leq C_{19} \int_{r>2 m+t}\left|u_1\right|^2+\left|\nabla u_0\right|^2+\left|u_0\right|^2+2 G\left(\left|u_0\right|\right) d x .
\end{align}
From estimates (4.18)--(4.21), (4.6) can be written as
\begin{align*}
 \operatorname{Re} \int_{R^3} u_t \phi_m \bar{u}_r d x&-\operatorname{Re} \int_{R^3} u_t(0) \phi_m(r, 0) \bar{u}_r(0) d x \\
& =\int_0^t \int_{r<m}-\frac{3}{2}\left|u_t\right|^2+\frac{1}{2}|\nabla u|^2+\frac{3}{2}|u|^2+3 G(|u|) d x d s \\
& +\int_0^t \int_{m<r<2 m+t} \frac{1}{2}|\nabla u|^2+\frac{3}{2}|u|^2+3 G(|u|) d x d s \\
& +C_{19} \int_0^t \int_{r>2 m+t}\left|u_1\right|^2+\left|\nabla u_0\right|^2+\left|u_0\right|^2+2 G\left(\left|u_0\right|\right) d x d s \\
& +\int_0^t \int_{m<r<2 m+t} \frac{1}{2 \operatorname{In} m}|\nabla u|^2 d x d s \\
& =I_1+I_2+I_3+I_4 .
\end{align*}
According to $K(u(t))<-\varepsilon$, we get $I_1<0$ and
\begin{align*}
I_2 & =\int_0^t \int_{m<r<2 m+t} \frac{1}{2}|\nabla u|^2+\frac{3}{2}|u|^2+3 G(|u|) d x d s \\
& =\int_0^t K_0(u(s)) d s \leq-\varepsilon t .
\end{align*}
When $m$ is large enough, we get
\begin{align*}
I_3=C_{19} \int_0^t \int_{r>2 m+t}\left|u_1\right|^2+\left|\nabla u_0\right|^2+\left|u_0\right|^2+2 G\left(\left|u_0\right|\right) d x d s \leq \int_0^t \frac{\varepsilon}{4} d s \leq \frac{\varepsilon t}{4},
\end{align*}
\begin{align*}
I_4=\int_0^t \int_{m<r} \frac{1}{2 \operatorname{In} m}|\nabla u|^2 d x d s \leq \int_0^t \frac{\varepsilon}{4} d s \leq \frac{\varepsilon t}{4},
\end{align*}
so
\begin{align}\label{5.1.9}
\operatorname{Re} \int_{R^3} u_t \phi_m \bar{u}_r d x-\operatorname{Re} \int_{R^3} u_t(0) \phi_m(r, 0) \bar{u}_r(0) d x \leq-\frac{\varepsilon t}{2}.
\end{align}
In addition, when $t$ is sufficiently large, we can get
 \begin{align}
 \left|\phi_m(r, t)\right| \leq \frac{C_{20} t}{\operatorname{In} t}.
 \end{align}
Thus, (4.26) can be written as
\begin{align}\label{5.1.12}
\int_{R^3} \nabla u \cdot u_t d x \geq C(\varepsilon) \operatorname{In} t.
\end{align}
By H?lder's  inequality, we have
 \begin{align}\label{5.1.13}
 \left(\int_{R^3}|\nabla u|^2 d x\right)^{\frac{1}{2}} \cdot\left(\int_{R^3}\left|u_t\right|^2 d x\right)^{\frac{1}{2}} \geq C(\varepsilon) \operatorname{In} t,
 \end{align}
which contradicts the assumption that $u(t,x)$ is bounded in $H_r^1 (R^3 )$, so there is a sequence $\left\{t_n\right\}$, when $t_n \rightarrow \infty$,
$\left\|u\left(t_n\right)\right\|_{H_r^1} \rightarrow \infty$.

Take $u_0(x)=\phi\left(\frac{x}{\lambda}\right)$, $u_1(x)=0$, $\lambda>1$, where $\phi(x)$ is the ground state solution of (2.3), we obtain
\begin{align}
\left\|\left(u_0, u_1\right)-\left(\phi, 0\right)\right\|_{H_r^1 \times L_r^2}<\varepsilon.
\end{align}
Then we prove $\left(u_0, u_1\right) \in R_1$, from the proof process of Lemma 3.2, we can get that
$J_0\left(u_\lambda(x)\right)$ is a decreasing function of  $\lambda$($\lambda>1$), so
\begin{align}\label{5.1.10}
E\left(u_0, u_1\right)=E\left(u_0, 0\right)=J_0\left(u_0(x)\right)=J_0\left(\phi\left(\frac{x}{\lambda}\right)\right)<J_0(\phi(x))=d(0) ,
\end{align}
\begin{align}\label{5.1.11}
\frac{1}{3} \int_{R^3}\left|\nabla u_0\right|^2 d x=\frac{1}{3} \int_{R^3}\left|\nabla \phi\left(\frac{x}{\lambda}\right)\right|^2 d x=\frac{\lambda}{3} \int_{R^3}|\nabla \phi(x)|^2 d x>\frac{1}{3} \int_{R^3}|\nabla \phi(x)|^2 d x=d(0).
\end{align}
(4.31) and (4.32)  indicate $\left(u_0, u_1\right) \in R_1$, so we get Theorem 2.2. Then the instability of the standing wave solutions is obtained .

\section*{Acknowledgment} The first author is supported in part by the
National Science Foundation of China, grant 11971166.

\end{document}